\documentclass[12pt]{amsart}
\usepackage{amsmath,amsfonts,amsthm,amssymb,comment}
\usepackage{enumerate}
\usepackage{mathrsfs}
\usepackage{amscd}
\usepackage{times,fancyhdr}
\usepackage{array}
\usepackage{epsfig}
\usepackage{graphicx}
\usepackage{color}
\usepackage{mathtools}
\usepackage{hyperref}
\usepackage{enumitem}
\usepackage{tikz-cd}
\usepackage{graphicx}
\usepackage{tikz}
\usetikzlibrary{decorations.fractals}
\usetikzlibrary{lindenmayersystems}
\usepackage[matrix,arrow,ps]{xy}
 \UsePSspecials{dvips}
\usetikzlibrary{positioning}
\input{xy}

\usepackage{euscript}

\usepackage{footnote}
\makesavenoteenv{tabular}
\makesavenoteenv{table}

\usepackage{multirow}
\usepackage{multicol}

\newcounter{thecounter}
\numberwithin{thecounter}{section}
\newtheorem{lemma}[thecounter]{Lemma}
\newtheorem{proposition}[thecounter]{Proposition}
\newtheorem{theorem}[thecounter]{Theorem}
\newtheorem*{theorem*}{Theorem}
\newtheorem{thm}[thecounter]{Theorem}
\newtheorem{corollary}[thecounter]{Corollary}
\newtheorem{defn}[thecounter]{Definition}

\DeclareMathOperator{\End}{End}
\DeclareMathOperator{\ad}{ad}

\DeclareMathOperator{\re}{{re}}
\DeclareMathOperator{\im}{{im}}

\renewcommand{\ker}{\mathrm{Ker}}

\DeclareMathOperator{\ind}{{Ind}}

\DeclareMathOperator{\rank}{{rank}}

\DeclareMathOperator{\supp}{{supp}}

\newcommand{\comm}[1]{}

\renewcommand{\a}{\alpha}
\renewcommand{\b}{\beta}
\newcommand{\y}{\gamma}
\newcommand{\g}{\mathfrak g}
\newcommand{\h}{\mathfrak h}
\newcommand{\n}{\mathfrak n}
\newcommand{\en}{\begin{enumerate}}
\newcommand{\te}{\end{enumerate}}
\newcommand{\Z}{{\mathbb Z}}
\newcommand{\R}{{\mathbb R}}
\newcommand{\C}{{\mathbb C}}
\newcommand{\N}{{\mathbb N}}
\newcommand{\Q}{{\mathbb Q}}
\newcommand{\la}{\langle}
\newcommand{\ra}{\rangle}
\newcommand{\tg}{\tilde{\mathfrak{g}}}
\newcommand{\tn}{\tilde{\mathfrak{n}}}
\newcommand{\tf}{\tilde{f}}
\newcommand{\Gp}{\mathfrak{g}^+}

\newcommand{\eps}{\varepsilon}

\usepackage{fullpage}

\author{Lisa Carbone}
\email{lisa.carbone@rutgers.edu}
\author{Terence Coelho}
\email{tjc210@scarletmail.rutgers.edu}
\author{Scott H.\ Murray}
\email{scotthmurray@gmail.com}
\author{Forrest Thurman}
\email{fbt7@math.rutgers.edu}
\author{Songhao Zhu}
\email{zhu.math@gatech.edu}

\title{Growth of root multiplicities along imaginary \\root strings in Kac--Moody algebras}

\begin{document}
    
\begin{abstract}
Let $\mathfrak{g}$ be a symmetrizable Kac--Moody algebra. Given a root $\alpha$ and a real root $\beta$ of $\mathfrak{g}$, it is known that the $\beta$-string through $\alpha$, denoted $R_\alpha(\beta)$, is finite. Given an imaginary root $\beta$, we show that $R_\alpha(\beta)=\{\a\}$ or $R_\alpha(\beta)$ is infinite. If $(\beta,\beta)<0$, we also show that the multiplicity of the root ${\alpha+n\beta}$ grows at least exponentially as $n\to\infty$. If $(\b,\b)=(\alpha, \beta) = 0$, we show that $R_\alpha(\beta)$ is bi-infinite and the multiplicities of $\a+n\beta$ are bounded. 
If $(\b,\b)=0$ and $(\alpha, \beta) \neq 0$, we show that $R_\alpha(\beta)$ is semi-infinite and the multiplicity of $\a+n\b$ or $\a-n\b$ grows faster than every polynomial as $n\to\infty$. 
We also prove that  $\dim \g_{\a+\b} \geq \dim \g_\a + \dim \g_\b -1$ whenever  $\a \neq \b$ with $(\a, \b)<0$.
\end{abstract}

\maketitle


\section{Introduction}\label{Intro}
Let $\g = \g (A)$ be the Kac--Moody algebra of a symmetrizable generalized Cartan matrix~$A = (A_{ij})$ with Cartan subalgebra $\h$. 
For $\a\in\h^*$, define
\[\g_\a = \{x\in \g \mid [h, x]=\a(h)x \text{ for all } h\in \h\}.\]
If $\a\ne0$ and $\g_\a \ne0$, we call $\a$ a \emph{root} and $\g_\a$ the \emph{root space}.
The set of all roots $\Delta$ is called the \emph{root system} of~$\g$. Note that $0$ is not considered a root, but $\g_0=\h$.
We set $\Bar{\Delta} = \Delta \cup \{0\}$. 
The \emph{root space decomposition} is 
\[
\g = \h \oplus \bigoplus_{\a \in \Delta} \g_\a.
\]
When $A$ is positive definite, $\g$ is a finite dimensional semisimple Lie algebra. In such cases, the length of a root string is at most 4 and this maximum  is achieved  for $\g$ of Cartan type~$\text{G}_2$ \cite{Hu}. 
When $A$ is not positive definite, $\g$ is infinite dimensional. 

Since $A$ is symmetrizable, $\h^*$ can be equipped with a symmetric bilinear form as in \cite[Chapter~16]{Ca}, which induces a norm. 
Recall that a root is  {real} when its norm is positive; and {imaginary} otherwise.
Let $\Delta_{\re}$ and $\Delta_{\im}$ be the sets of real and imaginary roots respectively.
 For $\a\in \bar{\Delta}$ and $\b\in \Delta$, let $S_\a(\b)\subseteq \Z$ be the maximal set of consecutive integers including $0$ such that $\a+i\b\in \bar{\Delta}$ for every $i\in S_\a(\b)$.
Then the \textit{$\b$-root string through $\a$} \cite{Ca,K} is
\[
R_\a(\b):=\{\a+i\b\mid {i\in S_\a(\b)}\}\subseteq\\\bar{\Delta}.
\]
The behavior of root strings with $\b$ an imaginary root has not been widely studied.
Throughout the paper, we assume that $\N$ includes 0.
Since $S_\a(\b)$ is a set of consecutive integers, exactly one of the following holds:
\begin{itemize}
    \item $S_\a(\b)$ is finite, in which case $R_\a(\b)$ is finite;
    \item $\N\subseteq S_\a(\b)$, $-\N\not\subseteq S_\a(\b)$, in which case we call $R_\a(\b)$ \emph{semi-infinite in the direction of $\b$};
    \item $\N\not\subseteq S_\a(\b)$, $-\N\subseteq S_\a(\b)$, in which case we call $R_\a(\b)$  \emph{semi-infinite in the direction of~$-\b$};
    \item $S_\a(\b)=\Z$, in which case we call $R_\a(\b)$  \emph{bi-infinite}.
\end{itemize}
We call $R_\a(\b)$ \emph{trivial} if $|R_\a(\b)|=1$ (i.e., $R_\a(\b)=\{\a\}$).

Recall that $\dim \mathfrak{g}_\a$ is called the \emph{multiplicity of $\a$}.
By the \textit{growth along $R_\alpha(\beta)$ in the direction of $\b$},
we mean the growth of the multiplicity of ${\a+n\b}$ as $n\to\infty$. For detailed definitions regarding growth of functions, we refer to \cite{KL} and Section~\ref{subsec:RS&G} below. 

Every root string in the direction of a real root is finite \cite{K}. Moreover, the root multiplicity $\dim \mathfrak g_{\alpha}$ is 1 if $\alpha$ is a real root \cite{K}.
This agrees with the finite dimensional case, where every root is real.
Morita \cite{Mo2} showed that $A_{ij}=-1$ and $A_{ji}<-1$ for some $i,j$ if and only if 
$|\{\a+k\b:k\in\Z \}\cap\Delta_{\re}|\in\{3,4\}$
for some $(\a, \b) \in \Delta \times \Delta_{\re}$. However, many other questions about root multiplicities and the structure of root strings have remained unanswered.

In this work, we characterize the properties of root strings in the direction of an imaginary root and show that these root strings are infinite or trivial. Our motivation is to study  the behavior of the dimensions of imaginary root spaces along infinite root strings. It is known \cite{KacGrowth} that non-affine Kac--Moody algebras have infinite
Gelfand--Kirillov dimension and that the dimensions of the root spaces
of rank 2 hyperbolic Kac--Moody algebras grow exponentially \cite{Meu}. 
Suppose $\a\in\bar\Delta$ and $\b\in\Delta_{\im}$.
Since $\b$ is imaginary, $(\b, \b)\le 0$. If $(\b, \b) = 0$ (respectively $(\b, \b)<0$), $\b$ is said to be \textit{isotropic} (respectively \textit{non-isotropic}).
We assume $R_\a(\b)$ is nontrivial, that is, $R_\a(\b)$ contains a root other than $\a$.
Table~\ref{tab:results} summarizes our main results. For $(\b, \b) < 0$, our methods do not give an easy criterion to distinguish between  bi-infinite and semi-infinite root strings.
\begin{table}[ht]
\renewcommand{\arraystretch}{1.2}
\begin{tabular}{
>{\centering\arraybackslash}p{65pt}
>{\centering\arraybackslash}p{55pt}
>{\centering\arraybackslash}p{75pt}
>{\centering\arraybackslash}p{140pt}
>{\centering\arraybackslash}p{70pt}
}
\hline
\multicolumn{2}{c}{\textbf{Conditions}} & \textbf{$R_\a(\b)$} & \textbf{Growth} & \textbf{Reference}\\
\hline
\multicolumn{2}{c}{$(\b, \b) < 0$} & infinite & at least exponential &  Section~\ref{sec:nonIso} \\
\hline
\multirow{2}{*}{$(\b, \b)=0$} & $(\a, \b) = 0$ & bi-infinite & bounded & Subsection~\ref{subsec:0prod} \\
\cline{2-5}
 & $(\a, \b) \neq 0$ & semi-infinite & greater than polynomial &  Subsection~\ref{subsec:negProd}\\
   \hline
\end{tabular}
\vspace{2mm}

\caption{Results on nontrivial root strings through $\a$ in the direction of $\b$}\label{tab:results}
\end{table}


We now discuss our results in more detail.
First consider the case when $\b$ is non-isotropic. By applying \cite[Corollary~C]{TM}, we show that, if $R_\a(\b)$ is nontrivial, then at least one of $\a\pm\N\b$ is contained in $R_\a(\b)$ (Proposition~\ref{prop:semiinfRS}).
Then, by a result of Kac \cite[Corollary~9.12]{K}, we show that there exists a free Lie subalgebra in $\bigoplus_{k\in \N} \g_{k\b}$. Free Lie algebras  possess exponential growth, which leads to the following.
\begin{thm}\label{thm:nonIsoExp} Let~$\Delta$ be the set of roots of a symmetrizable Kac--Moody algebra $\g$.
Let $\b\in\Delta_{\im}$ be non-isotropic. If $|R_\a(\b)|>1$, then at least one of $\a\pm\N\b$ is contained in $R_\a(\b)$. Thus $R_\a(\b)$ is infinite or bi-infinite.
Moreover, the growth along $R_\a(\b)$ has an exponential lower bound.
\end{thm}

When $\b$ is isotropic, \cite[Corollary~C]{TM} and \cite[Corollary~9.12]{K} no longer apply. However, we observe that when $(\a, \b) = 0$, the behavior of $R_\a(\b)$ is  essentially the same as in affine Kac--Moody algebras. 
This is established in Proposition~\ref{prop:equalDim}. 
\begin{thm} \label{thm:isoZero} 
Let~$\Delta$ be the set of roots of a symmetrizable Kac--Moody algebra $\g$. 
Let $\b \in \Delta_{\im}$ be isotropic, and $\a\in \Bar{\Delta}$ be such that $|R_\a(\b)|>1$. Suppose $(\a, \b) = 0$. Then $R_\a(\b)$ is always bi-infinite, and
\begin{enumerate}[label=(\roman*)]
\item If $\a$ is real, $R_\a(\b)$ consists only of real roots which have multiplicity $1$.
\item\label{thm:isoZero:im} If $\a$ is imaginary, $R_\a(\b)$ consists of imaginary roots only,
and the multiplicities take on at most $3$ values, at most $2$ of which occur periodically.
\end{enumerate}
\end{thm}
In particular, part \ref{thm:isoZero:im} is derived from a case-by-case study of untwisted and twisted affine Kac--Moody algebras. 

If $(\a, \b)\neq 0$, then reducing to the affine case is no longer applicable. However,  we can construct a subalgebra in $\bigoplus_{k\in \Z} \g_{k\b}$ that is isomorphic to an infinite dimensional Heisenberg algebra, which we denote  $H(\b)$. 
By applying the representation theory of $H(\b)$, we show that there exists a subspace in $\bigoplus_{k\in \N} \g_{\a+k\b}$ which is isomorphic to an irreducible induced module of $H(\b)$. 

\begin{thm}\label{thm:isoNeg} 
Let~$\Delta$ be the set of roots of a symmetrizable Kac--Moody algebra $\g$. 
Let $\b \in \Delta_{\im}$ be isotropic, and $\a\in \Bar{\Delta}$ be such that $|R_\a(\b)|>1$. Suppose $(\a, \b) \neq 0$. 
Then $R_\a(\b)$ is semi-infinite in the direction of $\pm\b$ and the multiplicity of $\a\pm n\b$ grows faster than every polynomial. 
\end{thm}

We remark that the potentially subexponential growth in Theorem~\ref{thm:isoNeg} arises from the growth of the partition function $p(n)$ which  naturally appears as a lower bound in the proof. The subexponential growth of $p(n)$  is a consequence of the classic Hardy--Ramanujan approximation formula \cite{HRpartFun} (see also Subsection~\ref{subsec:negProd}).




The following general result about root multiplicities may already be well-known. However, we have been unable to find a reference. Thus, we present a proof in Section~\ref{sec:linBound}.


\begin{thm}\label{thm:dimGrowth} Let~$\Delta$ be the set of roots of a symmetrizable Kac--Moody algebra $\g$. 
Let $\a$, $\b$ be distinct roots of $\g$ with  $(\a,\b)<0$. Then 
\[
\dim \g_{\a+\b}\ge \dim \g_\a+ \dim \g_\b -1.
\]
\end{thm}

\section{Preliminaries} \label{sec:pre}
In this section, we review some basic definitions and properties of Kac--Moody algebras. Most results here are contained in \cite{K} and~\cite{Ca}. 
Throughout this paper, we take $\N := \{0, 1, 2, \dots \}$.

Let $I$ be a finite index set. An integral square  matrix $A = (A_{ij})_{i, j \in I}$ is said to be a \textit{generalized Cartan matrix} if it satisfies:
\begin{align*}
A_{ij}&\in\Z &&\text{ for }i,j\in I;\\
A_{ii} &=2 &&\text{ for }i\in I;
\\
A_{ij} &\le 0\quad\text{ and }\quad
A_{ij} =0\iff A_{ji}=0 &&\text{ for }i, j\in I, i\neq j.
\end{align*}
A generalized Cartan matrix $A$ is called \textit{symmetrizable} if there exists a diagonal matrix $D$ over the  positive rationals such that the matrix $DA$ is symmetric. 
The matrix $D$ is called the \textit{symmetrizer} of $A$.
For $J\subseteq I$, we use $A_J$ to denote the submatrix $(A_{ij})_{i, j\in J}$ of $A$ with indices in $J$.


\subsection{Kac--Moody algebras}
For a generalized Cartan matrix $A$, let $\h=\h(A)$ be a $\C$-vector space  of dimension $\dim(\h)=2|I|-\rank(A)$. We choose linearly independent sets 
\begin{align*}
\Pi=\Pi_A=\{\a_{A,i}\}_{i\in I}=\{\a_{i}\}_{i\in I}&\subseteq \h^{\ast}\\
\Pi^\vee=\Pi^\vee_A=\{\a_{A,i}^\vee\}_{i\in I}=
\{\a_i^\vee\}_{i\in I}&\subseteq \h
\end{align*}
such that $\la \a_j, \a_i^{\vee} \ra=A_{ij}$ for all $i,j\in I$.
Here $\la \cdot, \cdot \ra$ denotes the natural pairing $\h^\ast \otimes \h \rightarrow \C$.
We call elements of $\Pi$ the \textit{simple roots} and of $\Pi^\vee$ the \textit{simple coroots}. Following \cite{Ca}, we let $\tg(A)$ be the Lie algebra generated by $\mathfrak{h}$, $\{\tilde{e}_{A,i}\}_{i\in I}$, and $\{\tilde{f}_{A,i}\}_{i\in I}$, subject to the relations
\begin{alignat*}{2}
h-ah'-bh'' &= 0, &&   \text{for $h, h', h''\in \h, a, b\in \C$  with $h = ah'+bh''$ in $\mathfrak{h}$;} \\
[h, h'] &= 0, &&  \text{for }  h, h'\in \h;\\
[\tilde{e}_{A,i},\tilde{f}_{A,j}] &= \delta_{ij}\alpha_i^{\vee},  && \text{for }  i,j\in I;\\
[h,\tilde{e}_i] &= \la \a_i, h \ra \tilde{e}_i,\  && \text{for }  h\in \h, i\in I;\\
[h,\tilde{f}_i] &= -\la \a_i, h \ra \tilde{f}_i,\qquad  && \text{for }  h\in \h, i\in I.
\end{alignat*}
Note that $\mathfrak{h}$ can be considered an abelian subalgebra of $\tg(A)$.
We write $\tn^+(A)$ (respectively $\tn^-(A)$) for the free subalgebra of $\tg(A)$ generated by the $\tilde{e}_i$ (respectively $\tf_i$). Then $\tg(A) = \tilde{\n}^-(A) \oplus \h\oplus \tilde{\n}^+(A)$ by \cite[Theorem~1.2]{K}.
The \textit{Kac--Moody algebra} is then defined as $ \g=\g(A) = \tg(A)/S(A)$ where $S(A)$ is the unique maximal ideal intersecting $\h$ trivially. 
We set
\begin{alignat}{3}
    S^+(A) &:=S(A)\cap \tn^+(A), \quad &S^-(A) &:=S(A)\cap \tn^-(A)  \label{eqn:S+-}\\
     \n^+=\n^+(A) &:=\tn^+(A)/S^+(A), \quad & \n^-=\n^-(A) &:=\tn^-(A)/S^-(A)  \label{eqn:n+-}
\end{alignat}
Then $S(A) = S^+(A)\oplus S^-(A)$, and we have the \textit{triangular decomposition}
\[
\g(A) = \n^-(A) \oplus \h(A) \oplus \n^+(A).
\]

The images of $\tilde{e}_{A,i}$ and $\tilde{f}_{A,i}$ in $\g(A)$ are denoted by $e_{A,i}=e_i$ and $f_{A,i}=f_i$, respectively. Note that the subalgebras $\n^+(A)$ and $\n^-(A)$ in $\g(A)$ are generated by $e_i$ and $f_i$ for $i\in I$.
\comm{For $\a\in \h^\ast$, we set $\g_{\a}=\{x\in\g \mid [h,x]=\la \a, h \ra x,\ h\in\h\}$.
We say $ \a\in \h^\ast\setminus\{0\}$ is a \textit{root} if $\g_{\a}\ne \{0\}$. }
We have $\g_{\a_i}=\C e_i, \; \g_{-\a_i}=\C f_i$, and write $\g_0=\h$. Thus $\g = \bigoplus_{\a\in \bar{\Delta}} \g_\a$.
\comm{We set $\bar{\Delta}=\Delta\cup\{0\}$. 
Thus the root space decomposition is $\g=\h\oplus \bigoplus\limits_{\a\in \Delta} \g_\a = \bigoplus\limits_{\a\in \bar{\Delta}} \g_\a$.}
For $\a\in \Delta$, we also have
\begin{equation} \label{eqn:dim-=+}
\dim \g_\a = \dim \g_{-\a}.
\end{equation}

The \textit{root lattice} $Q=Q_A$ in $\h^*$ is $\Z\Pi=\bigoplus_{\a_i\in\Pi}\Z\a_i$.
We set
\[
Q^+=Q^+_A :=(\Z_{\geq 0}\Pi) \setminus \{0\},\quad 
Q^-=Q^-_A :=-Q^+.
\]
Let $\Delta^+=\Delta\cap Q^+$ and $\Delta^-=\Delta\cap Q^-$, so $\Delta = \Delta^+\cup \Delta^-$ with $\Delta^- = -\Delta^+$. The sets $\Delta^{\pm}$ are called the sets of \textit{positive} and \textit{negative} roots respectively and we set $\bar{\Delta}^+=\Delta^+\cup\{0\}$. 


\subsection{Weyl group}
The \textit{Weyl group} is the subgroup of $\mathrm{GL}(\h^*)$ generated by the \textit{simple reflections}
\[ 
w_i=w_{i,A}: \lambda \mapsto \lambda - \la \lambda, \a_i^{\vee} \ra \a_i
\]
for $i\in I$.
The contragredient action of $W$ on $\h$ is given by $\la \lambda, w_i(h)\ra =\la w_i(\lambda), h\ra$ for $i\in I, \lambda\in \h^{\ast}, h\in \h$. Explicitly, 
\[
w_i(h)=h-\la \a_i, h \ra \a_i^\vee.
\]
Given $x\in \g$, we write $\ad x: y\mapsto [x, y]$ for $y\in \g$.
Then for $i\in I$, $\ad e_i$ and $\ad f_i$ are locally nilpotent on $\mathfrak{g}$ and the following automorphism of $\g$ is well-defined by \cite[Proposition 3.4]{Ca}
\[
\widetilde{w}_{i}=\exp(\ad e_i)
\exp(\ad (-f_i))
\exp(\ad e_i).
\]
For $\b\in\bar{\Delta}$, $i\in I$, we have $\widetilde{w}_{i}(\g_\b) =\g_{w_i(\b)}$, and $\widetilde{w}_{i}(h)=w_i(h)$ for $h\in\h$ \cite[Propositions~16.14 and~16.15]{Ca}.

A root $\a\in\Delta$ is called a \emph{real root} if 
there exists $w\in W$ such that $w\a\in \Pi$. 
A root which is not real is called \emph{imaginary}.  
We denote the real roots by $\Delta_{\re}$ and the imaginary roots by $\Delta_{\im}$. Thus $\Delta = \Delta_{\re} \cup \Delta_{\im}$.  
We write $\Delta^{\pm}_{\im} =\Delta_{\im}\cap \Delta^\pm$ and $\Delta^{\pm}_{\re}=\Delta_{\re}\cap \Delta^\pm$.

We summarize some properties of the Weyl group. For a complete treatment, we refer to \cite[Chapters~3 and~5]{K}  and \cite[Chapter~16]{Ca}.

\begin{proposition} \label{prop:Wprops}
    In the above notation:
    \begin{enumerate}[label=(\roman*)]
        
        \item $W$ preserves $\Delta_{\re}$, $\Delta_{\im}^+$, and $\Delta_{\im}^-$.
        \item $\dim \g_{\a}=\dim \g_{w(\a)}$ for $w\in W$, $\a\in \bar{\Delta}$.
        \item $\dim \g_\a = 1$ for $\a\in \Delta_{\re}$.
        \item For $\a\in \bar{\Delta}$ and $\b\in \Delta_{\re}$, $\{\a+\Z\b\}\cap \bar{\Delta}$ is finite.
    \end{enumerate}
\end{proposition}

For $\y\in Q$, the \textit{support} of $\y$, denoted as $\supp(\y)$, is the subset of $I$ for which the coefficient of $\a_i$ is non-zero in the expression of $\y$ as a linear combination of simple roots. Recall that the Dynkin subdiagram corresponding to  an index subset $J\subseteq I$ is the subdiagram of the Dynkin diagram of~$\Delta$ consisting of vertices $j\in J$ and all edges connecting these vertices.
We say $\y$ is \textit{connected} if the Dynkin subdiagram corresponding to $\supp(\y)$ is connected.
Then \cite[Proposition~16.21]{Ca} shows that all roots are connected in  this sense.

Let $C=C(A)\subset \h^\ast$ denote the elements $\y$ with $\la \y, \a_i^\vee \ra> 0$ for all $i\in I$. Let $\bar{C}$ denote its closure. We will be primarily interested in $-\bar{C}$, consisting of $\gamma \in \mathfrak{h}^\ast$ such that $\la \y, \a_i^\vee \ra \le 0$ for all $i\in I$. Let
\begin{equation} \label{eqn:KA}
    K = K_A := \{\gamma \in Q^+\cap -\bar{C} \mid \gamma \text{ is connected.}\}.
\end{equation}
We will need the following results. 

\begin{proposition} {\cite[Theorem~5.4]{K}}\label{prop:K_A}
The orbits of $K$ under the action of the Weyl group cover $\Delta^+_{\im}$. That is,
\[
\Delta^+_{\im}=\bigcup\limits_{w\in W} w K.
\]
\end{proposition}
\begin{proposition} {\cite[Proposition~5.5]{K}}\label{prop:propIm}
For $\a\in \Delta_{\im}$ and $r\in \mathbb{Q}\setminus \{0\}$ such 
that $r\a\in Q$, we have $r\a\in \Delta_{\im}$.
\end{proposition}

\subsection{Symmetric invariant bilinear form}
Let $A$ be a symmetrizable generalized Cartan matrix with symmetrizer~$D$ and write $q_i=D_{ii}\in\Q$.
Then we can define a non-degenerate symmetric form $(\cdot,\cdot)$ on $\h^*$ \cite[Lemma~2.1]{K} such that
\begin{equation} \label{eqn:()<>}
    (\a_i, \y) =q_i \la \y, \a_i^\vee \ra
\end{equation}
for  $i\in I$ and $\y\in \h^\ast$.
The form $(\cdot, \cdot)$ induces a bijection $\Lambda:\h\rightarrow \h^\ast$ such that for $\y\in \h^\ast, h\in \h$, $\la \y, h \ra=(\y,\Lambda(h))$.
By definition of $(\cdot,\cdot)$ on $\h^\ast$, we have $\Lambda(\a_i^\vee)=q_i^{-1} \a_i$. Thus for $i,j\in I$, we have
\begin{equation} \label{eqn:qiAij}
    (\a_i,\a_j)= q_iA_{ij}.
\end{equation}
This implies $(\a_i, \a_i)>0$, that is, the squared length of a simple root is positive.

The bijection $\Lambda$ induces a non-degenerate symmetric form on $\h$ which is also denoted  $(\cdot,\cdot)$.
This form on $\h$ lifts uniquely to a non-degenerate, symmetric, bilinear and invariant form $(\cdot,\cdot)$ on all of $\g(A)$ (\cite[Theorem~2.2]{K}) in the sense that  $([x,y],z)=(x,[y,z])$ for all $x,y,z\in \g(A)$.

We summarize  properties of the form $(\cdot, \cdot)$ on  $\h^*$ and $\h$:
\begin{proposition} \cite[Propositions~3.9 \& 5.2(c)]{K}\label{prop:(h*)}
    Let $(\cdot, \cdot)$ be the form on $\h^*$ as above. 
    \begin{enumerate}[label=(\roman*)]
        \item $(\cdot, \cdot)$ is $W$-invariant, that is, $(\gamma, \delta) = (w(\gamma), w(\delta))$ for all $w\in W$ and $\gamma, \delta\in \h^*$.
        \label{prop:(h*)inv}
        \item For all $\a\in \Delta$, $\a$ is imaginary if and only if $(\a, \a)\le 0$.\label{prop:(h*)im}
    \end{enumerate}
\end{proposition}

By Proposition~\ref{prop:(h*)}(ii) above, if $\a\in \Delta_{\re}$, then $(\a, \a)>0$. For $\b \in \Delta_{\im}$, if $(\b,\b)=0$ then $\b$ is called {\it isotropic}, and otherwise $\b$ is called {\it non-isotropic}.

\begin{proposition}\label{prop:(h)} \cite[Theorem~2.2]{K}
Let $(\cdot, \cdot)$ be the form on $\g$ as above. 
\en [label=(\roman*)]
\item If $\a +\b\neq 0$, then $(\g_\a,\g_\b)=0$. If $\a\in\Delta$,  then $(\cdot,\cdot)|_{\g_\a+\g_{-\a}}$ is non-degenerate.
\item For $x\in \g_\a$, $y\in \g_{-\a}$, $[x,y]=(x,y)\Lambda^{-1}(\a)$.
\te
\end{proposition}
We will also use the following results.

\begin{theorem}\label{thm:MarThmA} \cite[Theorem~A]{TM}
Let $A$ be a symmetrizable generalized Cartan matrix and $\g=\g(A)$.
Let $\b_1,\b_2\in {\Delta}$ be distinct with $(\b_1,\b_2)<0$. Then for all non-zero $x_1\in \g_{\b_1}$ and  $x_2 \in \g_{\b_2}$, we have $[x_1,x_2]\ne 0$.
\end{theorem}

\begin{corollary}\label{cor:MarCorC} \cite[Corollary~C]{TM}
Under the  conditions of Theorem~\ref{thm:MarThmA},
\[
\dim \g_{\b_1+\b_2}\ge {\rm max} (\dim \g_{\b_1}, \dim \g_{\b_2}).
\]
\end{corollary}

\section{Root strings and growth} \label{subsec:RS&G}
Let $A$ be a symmetrizable generalized Cartan matrix and $\Delta=\Delta(A)$. For $\a\in \bar{\Delta}$ and $\b\in \Delta$, let $S_\a(\b)\subset \Z$ be the maximal consecutive subset of the integers including $0$ such that $\a+i\b\in \bar{\Delta}$ for $i\in S_\a(\b)$.
Then the \textit{$\b$-root string through $\a$} is (\cite{Ca}, \cite{K}):
\[
R_\a(\b):=\{\a+i\b\mid {i\in S_\a(\b)}\}\subseteq\\\bar{\Delta}.
\]
Since $0\in \bar{\Delta}$, we may have $\a=0$. In this case, $R_0(\b)$ consists of all roots that are integral multiples of $\b$.
By Proposition~\ref{prop:Wprops}(ii), we have
\begin{equation} \label{eqn:wR=Rw}
    R_\a(-\b) = R_\a(\b), \qquad  w(R_\a(\b))= R_{w(\a)}(w(\b))
\end{equation}
for $w\in W$.
We also have
\begin{equation} \label{eqn:R=-R}
    R_{-\a}(\b)= -R_\a(\b),
\end{equation}
which follows from the fact that $\gamma\in \Delta$ if and only if $-\gamma\in \Delta$.
We will also consider the \textit{$\b$-root space through $\a$}
\begin{equation} \label{eqn:rtSp}
    \mathfrak{S}_\a(\b):=\bigoplus_{\gamma\in R_\a(\b)}
\g_\gamma.
\end{equation}

We recall the  definitions of bi-infinite and semi-infinite root strings from Section~\ref{Intro}.
We say $\gamma\in R_\a(\b)$ is an \emph{endpoint} of the root string if either $\gamma+\b\notin R_\a(\b)$ or $\gamma-\b\notin R_\a(\b)$.

By Proposition~\ref{prop:Wprops}(iv), if $\b\in \Delta_{\re}$, then $R_\a(\b)$ is finite.  In this case, if $\y$ is an endpoint of $R_\a(\b)$, it follows from \cite[Proposition~5.1]{K}  that $|R_\a(\b)|=|\langle \y, \b^\vee \rangle |+1$ when $\a$ and $\b$ are non-proportional.

Recall that if $\gamma$ is a real root,  $\dim \g_{\gamma}=1$.
We are interested in $\dim \g_{\gamma}$ when $\gamma$ is imaginary and of the form $\gamma=\a +i\b$ where $i\in\Z$, $\a\in\Delta$ and $\b\in \Delta_{\im}$. We will study the behavior of $\dim \g_{\a +i\b}$ along the root string $R_\a(\b)$  in the direction of $\b$. That is, we will consider $\dim \g_{\a+i\b}$ as $i\to\infty$. 
\comm{
For a function $f: \N \rightarrow \R$, we say that $f(n)$ is \textit{exponential in $n$} if there exists $b>1$ and $N>0$ such that $f(n) > b^n$ for $n>N$; that $f(n)$ is \textit{superpolynomial in $n$} if for any polynomial $p(n)$, there exists $N>0$ such that $f(n)>p(n)$ for $n>N$; that $f(n)$ is \textit{subexponential in $n$} if it is superpolynomial, and for any $b>1$, there exists $N>0$ such that $f(n)<b^n$ for $n>N$ \cite{GP}.

{\color{gray} Following \cite{GP}, we give the following definitions. Given two functions $f, g:  \N \rightarrow \R$, if $f(n) \le Cg(\alpha n)$ for all $n$ and some $C, \alpha > 0$, then we write $f\preceq g$. If $f\preceq g$ and $g\preceq f$, we write $f \sim g$. Thus, a function $f$ is said to be \textit{exponential} if $f \sim e^n$. A function $f$ is said to be \textit{subexponential} if 
\[
\lim_{n\to \infty} \frac{\ln(f(n))}{n} = 0.
\]
}
}

Following \cite{KL}, we give the following definitions. Let $\Phi$ be the set of functions from $\N$ to $\R_{>0}$ which are eventually monotone. 
If $f, g\in \Phi$ and if $f(n) \le Cg(\alpha n)$ for large enough $n$ and some $C>0$ and some $\alpha > 0$, we write $f\preceq g$. When $f\preceq g$ and $g\preceq f$, we write $f \sim g$ which defines an equivalence relation. 
The \textit{growth} of $f$ is the equivalence class $\mathcal{G}(f)$ in $\Phi/\sim$ which inherits a partial ordering $\le$ from $\preceq$ on $\Phi$. We say $\mathcal{L} \in \Phi/\sim$ is a lower bound of $\mathcal{G}(f)$ if $\mathcal{L} \le \mathcal{G}(f)$.
Let $\mathcal{P}_k := \mathcal{G}(n^k)$ for $k\in \N$ and $\mathcal{E}_r := \mathcal{G}(e^{n^r})$ for $r\in \R_{>0}$. 
We say that $f$ has \textit{polynomial growth} if $\mathcal{G}(f) = \mathcal{P}_k$ for some $k \ge 1$; 
that $f$ has \textit{exponential growth} if $\mathcal{G}(f) = \mathcal{E}_1$; 
that $f$ has \textit{subexponential growth} if $\mathcal{G}(f)\le \mathcal{E}_1$ while $\mathcal{G}(f) \not \le \mathcal{P}_k$ for any $k\in \N$. 
In particular, we see that $\mathcal{G}(f) = \mathcal{E}_{1/2}$ implies that $f$ has subexponential growth.

\comm{
\begin{defn}
Let $\g$ be a Kac--Moody algebra with root system $\Delta$. 
For $\a\in \bar{\Delta}$ and $\b\in \Delta$, we say that $R_\a(\b)$ is \emph{exponential in $\b$} or \emph{has exponential growth in $\b$} if $\dim \g_{\a+n\b}$ is at least exponential in $n$;
that $R_\a(\b)$ is \emph{subexponential in $\b$}, or \emph{has subexponential growth in $\b$} if $\dim \g_{\a+n\b}$ is at least subexponential in $n$.
\end{defn} 
}


\begin{defn}  \label{defn:RG}
Let $\g$ be a Kac--Moody algebra with root system $\Delta$. 
For $\a\in \bar{\Delta}$ and $\b\in \Delta$, the \emph{growth along $R_\a(\b)$ in the direction of $\b$} is defined to be $\mathcal{G}(\dim \g_{\a+n\b})$. 
\end{defn}

\section{Root multiplicities and root strings for non-isotropic imaginary roots} \label{sec:nonIso}
In this section, we study the growth along $R_\a(\b)$ when $\b$ is imaginary and non-isotropic, that is,  $(\b,\b)<0$. 
Our approach uses properties of certain free subalgebras of $\mathfrak{g}$. The dimensions of the homogeneous components in a  free Lie algebra $L$ are given by Witt's formula \cite{Witt}. When~$L$ has two or more generators, this formula has an exponential 
lower bound (Proposition~\ref{prop:FLAExp}). We then observe that there must exist an $s>0$ such that $\dim \g_{s\b}$ is at least 2 (Proposition~\ref{prop:s=5}).
Corollary~\ref{cor:dim_exp} then indicates that $\dim \g_{n\b}$ grows at least exponentially in~$n$. Corollary~\ref{cor:gen_exp} shows that this is true for $R_\a(\b)$. 

The following result shows that a root string with at least two roots in the direction of a non-isotropic root is infinite.

\begin{proposition} \label{prop:semiinfRS}
Let $\b$ be an imaginary root with $(\b, \b)<0$ and let $\a\in \bar{\Delta}$. If $|R_\a(\b)|>1$, then at least one of $\a\pm\N\b$ is contained in $R_\a(\b)$.
\end{proposition}
\begin{proof}
If $\a,\b$ are proportional, the proposition holds by Proposition~\ref{prop:propIm}, so we may assume $\a$ and 
$\b$ are not proportional. If $(\a,\b)<0$ (respectively $>0$), then by inductively applying Corollary~\ref{cor:MarCorC}, we have $\a+\N\b \subseteq R_\a(\b)$ (respectively $\a-\N\b \subseteq R_\a(\b)$).
If $(\a,\b)=0$ but $\a+\b\in \Delta$ (respectively $\a-\b\in \Delta$), then since $(\a+\b,\b)=(\b,\b)<0$, we have $\a+\N\b\subseteq R_\a(\b)$ (respectively $\a-\N\b\subseteq R_\a(\b)$).
\end{proof}

Replacing $\b$ by $-\b$ if necessary (see Equations~(\ref{eqn:dim-=+}), (\ref{eqn:wR=Rw}), and~(\ref{eqn:R=-R})), we may assume that $\a +\N \b \subseteq R_\a(\b)$. 
In this section, we will prove Theorem~\ref{thm:nonIsoExp}.
\comm{
\begin{theorem}
Let $A$ be a symmetrizable generalized Cartan matrix with 
$\Delta=\Delta(A)$ and let $\b\in\Delta_{\im}$ with $(\b, \b)<0$. Let $\a\in \bar{\Delta}$ be such that $\a+\N\b\subset \bar{\Delta}$. Then $R_\a(\b)$ has exponential growth in $\b$.
\end{theorem}
}
\subsection{Free Lie algebras} \label{subsec:FLA}
We first present the necessary background on free Lie algebras. Our main sources are  \cite{Hu, ReutenauerFLA} and \cite{CasselmanFLA}.

Let $X$ be a set.
We denote the \textit{free Lie algebra generated by} $X$ by $\mathfrak{L}(X)$. 
A \textit{basic Lie monomial of degree} $n$ over $X$ is an arbitrary iterated Lie multi-bracket involving $n$ elements from $X$.
Then $$\mathfrak{L}(X)=\bigoplus_{n\ge 1} \mathfrak{L}_n(X)$$
where $\mathfrak{L}_n(X)$ is the is the homogeneous component of degree $n$ in $\mathfrak{L}(X)$ spanned by degree $n$ Lie monomials.

The canonical embedding of $X$ in $\mathfrak{L}(X)$ is denoted  $i_X : X\rightarrow \mathfrak{L}(X)$. By the universal property of $\mathfrak{L}(X)$, any set-map $f$ from $X$ to a Lie algebra $\mathfrak{m}$ lifts uniquely to a Lie algebra homomorphism $\tilde{f}$ from $\mathfrak{L}(X)$ to $\mathfrak{m}$ such that $f = \tilde{f}\circ i_X$. 

 
Let $\mathcal A(X)$ be the free non-associative algebra on $X$.
Then $\mathcal{A}(X)$ is graded by the length of words in the elements of $X$.
Let $\mathcal{I}$ be the two-sided ideal of $\mathcal A(X)$ generated by elements
$$
(xx) \quad\text{ and } \quad
((xy)z)+((yz)x)+((zx)y)
$$
for $x,y,z\in \mathcal A(X)$. The quotient algebra 
$\mathcal A(X)/\mathcal{I}=\{x+\mathcal{I}\mid x\in \mathcal A(X)\}$
with Lie bracket
$[x+\mathcal{I},y+\mathcal{I}]=(xy)+\mathcal{I}$
is a Lie algebra. The ideal $\mathcal{I} = \bigoplus_{n\ge 1}\mathcal{I}_n$ is homogeneous, and $\mathcal{A}(X)/\mathcal{I}$ inherits the grading from $\mathcal{A}(X)$.
For $n\ge 1$, any basic Lie monomial on $n$ cosets of the form $x+\mathcal{I}$ with $x\in X$ is in the $n$-th homogeneous component $\mathcal{A}(X)_{n}/\mathcal{I}_{n}$.
For the following proposition, we refer to \cite{CasselmanFLA}.

\begin{proposition} \label{prop:FLAhomo}
Let $\mathfrak{L}(X)$ be the free Lie algebra generated by the set $X$. Then~$\mathfrak{L}(X) \cong \mathcal{A}(X)/\mathcal{I}$. Furthermore, $\mathcal{A}_n(X)/\mathcal{I}_n$ is the $n$th homogeneous component $\mathfrak{L}_n(X)$ of $\mathfrak{L}(X)$.
\end{proposition}

Let $V_X$ denote the vector space generated by $X$. Recall \cite{Hu} that the free Lie algebra $ \mathfrak{L}(X)$ is isomorphic to the Lie subalgebra generated by $X$ in the tensor algebra $T(V_X) := \bigoplus_{n\in \N} V_X^{\otimes n}$.
The Lie bracket is given by 
\[
[x, y] := x\otimes y- y \otimes x.
\]
Therefore any iterated Lie bracket 
can be expanded as a linear combination of pure tensors.

The following result is central to the proofs in this section.

\begin{proposition} \label{prop:FLAExp}
    With the above notation, let  $m:=|X| = \dim \mathfrak{L}_1(X)$. If $m\ge 2$, then for all $\eps >0$, there exists  $N_\eps$ such that for $n>N_\eps$, $\dim \mathfrak{L}_n(X) > m^{(1-\eps)n}$.
\end{proposition}
\begin{proof}
    By a classic result of Witt \cite[Satz~3]{Witt}, we have 
\begin{equation} \label{eqn:Witt}
    \dim \mathfrak{L}_n(X)= \frac{1}{n}\sum\limits_{d|n}\mu(d)m^{n/d}
\end{equation}
for $n>0$, where $\mu$ denotes the M\"obius function. By definition, $\mu(d)\ge -1$, and there are at most $n/2$ terms in the summation besides $n$. Hence, for $m>1$, 
\[
\frac{1}{n}\sum\limits_{d|n}\mu(d)m^{n/d}\ge \frac{1}{n}\left( m^n-\sum\limits_{i=1}^{n/2}m^i \right)
= \frac{1}{n}\left( m^n-\frac{m^{n/2+1}-m}{m-1} \right)
\ge\frac{m^n}{2n}
\]
Now, for all $\eps>0$, we take $n$ sufficiently large so that 
\[
\frac{m^n}{2n}>m^{(1-\eps)n}
\]
which yields the result by Equation~(\ref{eqn:Witt}).
\end{proof}

\subsection{Proof of Theorem~\ref{thm:nonIsoExp}} \label{subsec:pfThmgen_exp}

For the proof of Theorem~\ref{thm:nonIsoExp}, we will use the following important result.

\begin{proposition}
     \label{prop:s=5}
Let $\b\in\Delta_{\im}$ with $(\b, \b)<0$. 
Then for some $1\leq s \leq 5$, $\dim \g_{s\b}\ge 2$.
\end{proposition}


\begin{proof}
We outline the strategy of the proof. By Proposition~\ref{prop:semiinfRS}, $\dim \g_{s\b} \ge 1$. 
Suppose $\dim \g_{s\b} = 1$ for $s = 1, 2, 3, 4$. Otherwise we are done. We claim  that $\dim \g_{5\b}$ is  at least 2. To do this, we prove that there must be two linearly independent Lie monomials in $\g_{5\b}$. This is because $\bigoplus_{k\ge 1} \g_{k\b}$ is free (Step 1) and it contains a Lie subalgebra generated by two particular elements which is again free (Step 2). We observe that the two Lie monomials on the two generators are of different degrees, which gives the result (Step 3).

\textit{Step 1.} By \cite[Corollary~9.12(a)]{K}, $\bigoplus_{k\ge 1}\g_{k\b}$ is a free Lie algebra $\mathfrak{L}(X)$. Here $X$ is a basis for $\bigoplus_{k\ge 1} \g_{k\b}^0$ where
\begin{align} \label{eqn:Kac9.12}
    \g_{k\b}^0 := \{x\in \g_{k\b}\mid \,&(x, y) = 0 \text{ for every } y \text{ in the subalgebra generated by } \\
    &\g_{-i\b} \text{ for } i=1, \dots, k-1\}. \notag
\end{align}
In particular, we see that $\g_{\b}^0 = \g_{\b}$. 
Next we consider $\g_{2\b}^0$. Note that $\dim \g_{-\b} = \dim \g_\b = 1$ by the assumption above. So the subalgebra generated by $\g_{-\b}$ is itself. 
By Proposition~\ref{prop:(h)}(i), since $2\b+(-\b) \neq 0$, we have $(x, y) = 0$ for any $x \in \g_{2\b}$ and $y\in \g_{-\b}$. Hence $\g_{2\b}^0 = \g_{2\b}$. 
Let $x_1 \neq 0$ be in $\g_\b$. Then $\g_{2\b}$ cannot be spanned by $[x_1, x_1]=0$. Let $x_2\neq 0$ be in $\g_{2\b}$. Then $x_1$ and $x_2$ are linearly independent. The elements $x_1$ and $x_2$  are the two generators we consider below.

\textit{Step 2.}If a set $M$ is contained in another set $N$, then $\mathfrak{L}(M) \subseteq \mathfrak{L}(N)$. This can be seen as follows. For $M \subseteq N$, we have $T(V_M) \subseteq T(V_N)$. As free associative algebras, tensor algebras are isomorphic to the universal enveloping algebras by universal properties. Thus the universal enveloping algebra $\mathfrak{U}(\mathfrak{L}(M))$ embeds into $\mathfrak{U}(\mathfrak{L}(N))$.
Now by definition of $\g_{k\b}^0$ we see that $x_1, x_2\in X$. Let $Y := \{x_1, x_2\} \subseteq X$. Then $Y$ generates a free Lie algebra $\mathfrak{L}(Y)$ which is embedded in $\mathfrak{L}(X) = \bigoplus_{k\ge 1} \g_{k\b}$.

\textit{Step 3.} The two Lie monomials $[x_2, [x_1, x_2]]$ and $[x_1, [x_1, [x_1, x_2]]]$ are both in $\g_{5\b}$. 
Both can be expanded as non-zero combinations of pure tensors on $x_1$ and $x_2$.
Also, they are linearly independent as they are contained in different homogeneous components of $\mathfrak{L}(Y)$, as in Proposition~\ref{prop:FLAhomo}. 
Therefore $\dim \g_{5\b} \ge 2$.
\end{proof}

For $\gamma \in\Delta_{\im}$ with $(\gamma, \gamma)<0$, let $\Gp(\gamma)\subseteq \g$ be the subalgebra generated by $\g_\gamma$ and for $k\ge 1$, let
\[
\Gp(\gamma)_k=\Gp(\gamma)\cap \g_{k\gamma}.
\]
If $V$ is an vector space, recall that $\mathfrak{L}(V):=\mathfrak{L}(B)$ for any basis $B$ of $V$.
\begin{lemma} \label{lem:g+free}
    In the above notation, $\Gp(\gamma) \cong \mathfrak{L}(\g_\gamma)$.
\end{lemma}
\begin{proof}
    We apply the argument in Steps 1 and 2 in the above proof of Proposition~\ref{prop:s=5}. In particular, $\g_\gamma^0 = \g_\gamma$ as in Equation~(\ref{eqn:Kac9.12}) in Step 1. Replacing $Y$ by $\g_\gamma$ in Step 2, we see that $\mathfrak{L}(\g_\gamma)$ is a free subalgebra in $\bigoplus_{k\ge 1} \g_{k\gamma}$.
    Then $\Gp(\gamma) = \mathfrak{L}(\g_\gamma)$ as they have the same generators.
\end{proof}

As $[\Gp(\b)_m, \Gp(\b)_n]\subseteq \Gp(\b)_{m+n}$, the $\N$-grading on $\Gp(\b)$ above coincides with the grading on $\Gp(\b)$ as a free Lie algebra.

\begin{corollary}\label{cor:dim_exp}
Let $\b\in\Delta_{\im}$ with $(\b, \b)<0$. Suppose that for some $s>0$, $\dim \g_{s\b} \ge 2$. Then $\dim \g_{n\b}$ is at least exponential in $n$.
\end{corollary}
\begin{proof}
Let $\gamma \in \Delta_{\im}$ with $(\gamma, \gamma)<0$. By Lemma~\ref{lem:g+free} and Proposition~\ref{prop:FLAExp}, if $\dim \g_\gamma = m \ge 2$, then for all $\eps>0$, there exists $N$, such that for $n>N$, $\dim \Gp(\gamma)_n > m^{(1-\eps)n}$. 
It is then possible to first fix $\eps$, then choose $a>1$ satisfying $a^s = m^{1-\eps}$ and $N\in \N$, such that for $n>N$, $\dim \Gp(\gamma)_n > a^{sn}$.

Set $\gamma = s\b$. Fix $a$ as above. Then for $n>N$, $\dim \g_{n\gamma} = \dim \g_{ns\b} \ge \dim \Gp(\gamma)_n > a^{sn}$.
Using Corollary~\ref{cor:MarCorC}, and the fact that $(n\b,\b)<0$ for $n\ge 1$, it follows that $\dim \g_{n\b}$ is non-decreasing in $n$. 
Hence for any $0\le k \le s-1$ and $\ell > N$, we have
\begin{equation} \label{eqn:dim>=}
    \dim \g_{(\ell s+k)\b} \ge \dim \g_{\ell s\b}.
\end{equation}
But
\[
a^{\ell s} >  a^{\ell s+k-s} = a^{-s} a^{\ell s+k},
\]
which by Equation~(\ref{eqn:dim>=}) shows $\dim \g_{n\b} >  a^{-s}a^n$ when $n> s N$. 
Pick $b>1$ with $b<a$. Then for sufficiently large $n$,
$\dim \g_{n\b} >  b^n$. This shows $\dim \g_{n\b}$ is exponential in $n$.
\end{proof}

\begin{corollary}\label{cor:gen_exp}
Let $\b$ be as above. Let $\a\in \bar{\Delta}$ be such that $\a+\N\b\subseteq \bar{\Delta}$.
Then $\mathcal{G}(\dim \g_{\a+n\b}) \not \le \mathcal{E}_1$.
\end{corollary}

\begin{proof}
In Corollary~\ref{cor:dim_exp}, we showed that there exists $b>1$, and $N>0$ such that for $n>N$, $\dim \g_{n\b}> b^n$.
If $(\a,\b)<0$, then $(\a, n\b)<0$, and by Corollary~\ref{cor:MarCorC}, 
\[
\dim \g_{\a+n\b} \ge \dim \g_{n\b}> b^n \sim e^n
\]
for $n>N$.

Otherwise for $(\a, \b)\ge 0$, since $\a+\N \b \subseteq \bar{\Delta}$, we may take $n'$  to be the least natural number such that $(\a+n'\b,\b)<0$. Then for $n>N$, we have $\dim \g_{\a+(n+n')\b} > b^{n}$.
Equivalently, for $n> N+n'$
\[
\dim \g_{\a+n\b} > b^{n-n'} \sim e^n.
\]
That is, along $R_\a(\b)$, $\dim \g_{\a+n\b}$ has lower bound $\sim e^n$. Thus, $\mathcal{G}(\dim \g_{\a+n\b}) \not \le \mathcal{E}_1$.
\end{proof}

\begin{theorem*}[\ref{thm:nonIsoExp}]
Let~$\Delta$ be the set of roots of a symmetrizable Kac--Moody algebra $\g$.
Let $\b\in\Delta_{\im}$ be non-isotropic. If $|R_\a(\b)|>1$, then at least one of $\a\pm\N\b$ is contained in $R_\a(\b)$. Thus $R_\a(\b)$ is infinite or bi-infinite.
Moreover, the growth along $R_\a(\b)$ has an exponential lower bound.
\end{theorem*}
\begin{proof}

By Proposition~\ref{prop:semiinfRS}
at least one of $\a\pm\N\b$ is contained in $R_\a(\b)$.
By Proposition~\ref{prop:s=5}, $\dim \g_{5\b} \ge 2$. Take $s = 5$ in Corollary~\ref{cor:dim_exp}, and by Corollary~\ref{cor:gen_exp}, we have the desired result.
\end{proof}

We note the following special case.
\begin{corollary}
Let $\b\in\Delta_{\im}$ with $(\b, \b)<0$. For non-zero $x_1\in  \g_{\beta}$ and non-zero  $x_2\in\g_{2\beta}$, $x_1$ and $x_2$ generate a free Lie subalgebra in $\bigoplus_{k\ge 1}\g_{k\beta}$. 
\end{corollary}

\section{Root multiplicities and root strings for isotropic imaginary roots} \label{sec:Iso}
We now consider root string in the direction of an isotropic imaginary root. 
Let $\g=\g(A)$ for a symmetrizable generalized Cartan matrix $A$ as before.
Recall that $\b \in \Delta$ is said to be isotropic if $(\b, \b) = 0$.
Throughout this section, we assume that $\b$ is isotropic.
In this case, \cite[Corollary~C]{TM} and \cite[Corollary~9.12]{K} no longer apply, and a new approach is required.

In light of Proposition~\ref{prop:Wprops}(ii), Proposition~\ref{prop:K_A}, and Proposition~\ref{prop:(h*)}(i), it suffices to assume that $\a\in \Bar{\Delta}^+$ and $\b\in K$ (see Equation~(\ref{eqn:KA})) in our analysis of $R_\a(\b)$ and the  $\b$-root space $\mathfrak{S}_\a(\b)=\bigoplus_{\gamma\in R_\a(\b)}
\g_\gamma$ through $\a$ (see Equation~(\ref{eqn:rtSp})). 
In particular, $\b$ is assumed to be positive. 
Also, by definition of $K$ (see Equation~(\ref{eqn:KA})), we always have $(\a, \b)\le 0$.
Throughout the section, we assume that $|R_\a(\b)|>1$. 

We first divide the discussion into two cases: $(\a, \b) = 0$ and $(\a, \b) <0$. 
The first case requires some machinery related to affine Kac--Moody algebras, and the results in the second case are obtained by using the module theory of Heisenberg algebras. 

\begin{proposition}\label{prop:affineSupp}
Let $\mathfrak{g}$ be a symmetrizable Kac--Moody algebra with root system $\Delta$ and  $K$ as in Equation~(\ref{eqn:KA}).
Let $\b \in K$ be isotropic and let $\a\in \Bar{\Delta}^+$ be such that $|R_\a(\b)|>1$. Then $(\a, \b) = 0$ if and only if $\supp(\a) \subseteq \supp(\b)$.
\end{proposition}
\begin{proof}
Let $J = \supp(\b)$. By definition of $K$, $(\a, \b)\le 0$, so it suffices to show (i) if $\supp(\a)\subseteq J$, then $(\a, \b) = 0$;
and (ii) if $\supp(\a) \not \subseteq  J$, then $(\a, \b)<0$.

We first prove (i). By \cite[Proposition~16.29]{Ca}, the submatrix $A_J$ is of affine type. Moreover, by Equation~(\ref{eqn:qiAij}), the inner product defined on $Q_{A_J}$ agrees with the one on $Q_A$.
Then as an imaginary root of $\g(A_J)$, $\b$ is a multiple of the basic imaginary root $\delta$ (see Equation~(\ref{eqn:delta})). Thus by \cite[Section~17.1 Summary~(i)]{Ca},
\begin{equation} \label{eqn:aff0}
    (\a_i, \b) = 0, \quad \text{for }i\in J.
\end{equation}
If $\supp(\a)\subseteq J$, then $\a$ is an integral linear combination of $\a_i$ with $i\in J$, and $(\a,\b)=0$ by Equation~(\ref{eqn:aff0}). 

Now we prove (ii). We show that for at least one $j \in \supp(\a)$, we have $(\a_j, \b)<0$, and for all other $k\in \supp(\a)$, $(\a_k, \b)\le 0$. This would imply $(\a, \b) < 0$ as $\a \in \Delta^+$ is a positive integral linear combination of simple roots with indices in $\supp(\a)$.
Suppose $\supp(\a) \not \subseteq J$. 
Since $|R_\a(\b)|>1$, one of $\a\pm\b$ must be a root, hence $\supp(\a\pm \b) = \supp(\a) \cup J$ is connected by \cite[Proposition~16.21]{Ca}. 
Then there must be a $j\in \supp(\a)\setminus J$ that is a neighbor of some $i\in J$ in the Dynkin diagram.
By Equation~(\ref{eqn:()<>}), $ \la\alpha_i, \alpha_j^\vee \ra$ and $(\a_i, \a_j)$ have the same sign.
Since $ \la\alpha_i, \alpha_j^\vee \ra = A_{ij}<0$, we have $(\b,\a_j)<0$. 
For other $k\in \supp(\a)$, $k$ is either \emph{not} a neighbor of $J$, or $k\in J$.
The former gives $(\b, \a_k)\le 0$ as $A_{ik}\le 0$ for $i\neq k$, while the latter implies $(\b, \a_k) = 0$ by Equation~(\ref{eqn:aff0}). 
Therefore, $(\b, \a) \le (\b, \a_j) < 0$.
\end{proof}

\subsection{Zero inner product} \label{subsec:0prod}
We now assume that $(\a, \b) = 0$. Proposition~\ref{prop:affineSupp} tells us that $\supp(\a) \subseteq \supp(\b)=:J$. By \cite[Proposition~16.29]{Ca}, $A_J$ is of affine type.
In this subsection, we first equate root multiplicities of $\g(A)$ with those of $\g(A_J)$, for roots in the intersection of $\Delta(\g(A_J))$ and $\Delta(\g(A))$ (Proposition~\ref{prop:equalDim}), so that we need only consider the affine Kac--Moody algebra $\g(A_J)$.
Then by a collection of standard results regarding affine Kac--Moody algebras (Proposition~\ref{prop:affResults}), we derive the main result Theorem~\ref{thm:isoZero}.

Recall $\g(A) = \tg(A)/S(A)$ where $S(A)$ is the maximal ideal of $\g(A)$ that intersects $\h$ trivially.
For $i\ne j\in I$, let $S^+_{i,j}(A)$ (respectively $S^-_{i,j}(A)$) be the ideal of $\tg(A)$ generated by $(\ad(e_i))^{-A_{i,j}+1} e_j$ (respectively $(\ad(f_i))^{-A_{i,j}+1} f_j$).
The following is proven in detail in \cite[Theorem~19.30]{Ca}, assuming $A$ is symmetrizable (c.f. Equations~(\ref{eqn:S+-}) and~(\ref{eqn:n+-})):
\begin{equation} \label{eqn:S+-Sum}
    S^+(A)=\sum\limits_{i\ne j\in I} S^+_{i,j}(A), \quad S^-(A)=\sum\limits_{i\ne j\in I} S^-_{i,j}(A).
\end{equation}
We also point out that the ideals $S^+(A)$ and $S^-(A)$ are indeed graded by $Q^+$ and $Q^-$ respectively. See the proof of \cite[Theorem~19.30]{Ca}.
For $J\subseteq I$, we write $\Pi_J$ for $\Pi_{A_J}$, $\Delta_J^\pm$ for $\Delta^\pm(A_j)$ and $Q^\pm_J$ for $Q_{A_J}^\pm$, so $\Pi_J = \{\a_j\}_{j\in J} \subseteq \Pi$, $Q_J^+ = \N \Pi_J $, and $Q_J^- := -\N \Pi_J$ (see Section~\ref{sec:pre}).

\begin{proposition} \label{prop:equalDim}
Let $A = (A_{i,j})_{i, j\in I}$ be a symmetrizable generalized Cartan matrix, and $J \subseteq I$. 
Then for all $\b\in \Delta_J^+\cup \Delta_J^-$, we have 
\[
\dim(\g(A)_\b)=\dim(\g(A_J)_\b)
\]
where $\g(A_J)$ is the Kac--Moody algebra of the submatrix $A_J$.
\end{proposition} 
\begin{proof}
It suffices to consider $\b \in \Delta_J^+$.
Let $\tn^+_J$ be the free subalgebra of $\tn^+ := \tn^+(A)$ generated by $\{\tilde{e}_{A,j}\}_{j\in J}$.
By the universal property of free Lie algebras, the tautological map on generators $\tilde{e}_{A_J,j}\mapsto \tilde{e}_{A,j}$
uniquely lifts to a Lie algebra isomorphism
\begin{equation} \label{eqn:psiIso}
    \psi:\tn^+(A_J)\rightarrow \tn^+_J.
\end{equation}
Write $S^+$ for $S^+(A)$, and $S^+_{i,j}$ for $S^+_{i, j}(A)$.
Now for all $i,j\in J$, we have
$\psi(S_{i,j}^+(A_J))= S_{i,j}^+\cap \tn^+_J$. 
As in Equation~(\ref{eqn:S+-Sum}) above, we have
\[
S^+(A_J)=\sum_{i\ne j\in J} S_{i,j}^+(A_J)\subseteq \tn^+(A_J),\quad 
S^+=\sum_{i\ne j\in I} S_{i,j}^+ \subseteq \tn^+,
\]
and hence
\[
\psi(S^+(A_J))= \sum_{i\ne j\in J} \psi(S_{i,j}^+(A_J)) = \sum_{i\ne j\in J} S_{i, j}^+\cap \tn^+_J.
\]

By definition of $S^+_{i, j}$, we get $S_{i,j}^+\cap \tn^+_J=0$ if $i\notin J$ or $j\notin J$.
As a result, $\psi(S^+(A_J))= S^+\cap \tn^+_J$, and
$\psi$ (Equation~(\ref{eqn:psiIso})) descends to an isomorphism
\begin{equation} \label{eqn:psiBar}
    \bar{\psi} :\tn^+(A_J)/S^+(A_J) \rightarrow \tn^+_J/(S^+\cap \tn^+_J) \cong (\tn^+_J+S^+)/S^+.
\end{equation}
The left side of Equation~(\ref{eqn:psiBar}) is just $\n^+(A_J) = \bigoplus_{\b \in \Delta_J^+}\g(A_J)_\b$. The isomorphism on the right side of Equation~(\ref{eqn:psiBar}) is due to the Second Isomorphism Theorem. 

Consider $\n^+_J := \bigoplus_{\b\in \Delta_J^+} \g(A)_\b$. 
This is the subalgebra in the quotient $\n^+(A) = \tn^+/S^+$ generated by $\{e_{A, j}\}_{j\in J}$. 
The preimage of $e_{A, j}$ is $\tilde{e}_{A,j}$.
Consequently, the preimage of $\n^+_J$ is generated by $\{\tilde{e}_{A,j}\}$, equal to $\tn^+_J$ by definition.
Therefore $\n^+_J \subseteq (\tn^+_J+S^+)/S^+$, and $\bar{\psi}$ establishes an isomorphism between
$\n^+(A_J) = \bigoplus_{\b \in \Delta_J^+}\g(A_J)_\b$ and $\n^+_J := \bigoplus_{\b\in \Delta_J^+} \g(A)_\b$.
As all the ideals considered here are graded by $Q^+_J$, such an isomorphism is a graded isomorphism.
\end{proof}

Next, we summarize the results of roots and multiplicities for affine Kac--Moody algebras. 
We refer to a classic treatise by Carter \cite[Section~18]{Ca}, but c.f. \cite{K}. 
We borrow notations from \cite{Ca} below. Let $L^0$ be the finite semisimple Lie algebra associated with the generalized Cartan matrix $A$ of affine type, and $H^0$ the Cartan subalgebra of $L^0$. 
Let
\[
\hat{\mathfrak{L}}(L^0) := \C[t, t^{-1}]\otimes L^0 \oplus \C c \oplus\C d
\]
where $c$ is a certain central element and $d$ a derivation operator.
Let $H = (1 \otimes H^0) \oplus \C c \oplus \C d$.
Let $\delta\in H^*$ be defined by 
\begin{equation} \label{eqn:delta}
    \delta(H^0)=0,\quad \delta(c)=0,\quad \delta(d)=1. 
\end{equation}
The root  $\delta$ is called the basic imaginary root.
Then $\hat{\mathfrak{L}}(L^0)$ accounts for all the untwisted affine Kac--Moody algebras. By \cite[Theorem~18.15]{Ca}, $\h(A) = \h = H$.
Certain order 2 or 3 automorphisms $\tau$ are defined on $\hat{\mathfrak{L}}(L^0)$, and the fixed point subalgebras
$\hat{\mathfrak{L}}(L^0)^\tau$
account for the twisted affine Kac--Moody algebras.

\begin{proposition} \label{prop:affResults}
Let $\g = \g(A)$ where $A$ is a generalized Cartan matrix of affine type. Then
    \begin{enumerate}[label=(\roman*)]
        \item Roots strings consisting entirely of real roots are bi-infinite in the direction of $\delta$.
        \item Imaginary roots are of the form  $k\delta$ for $k\in \Z \setminus \{0\}$.
        \item Imaginary roots have at most two different multiplicities which occur periodically.
        \item For $k\neq -\ell$, $[\g_{k\delta}, \g_{\ell\delta}] = 0$.
    \end{enumerate}
\end{proposition}
\begin{proof}
Part (i) follows from \cite[Theorem~17.17]{Ca} which contains a complete list of real roots. In particular, a real root is always of the form of $x\a+ y\delta$ where $\a$ is a real root of $L^0$, $x$ is either 1 or $\frac{1}{2}$, and $y \in \Z$, $2\Z$, $3\Z$, or $\Z-\frac{1}{2}$.
Part (ii) follows from \cite[Theorem~16.27(ii)]{Ca}. 
For part (iii), in the untwisted case, we have $\g_{k\delta} = t^k \otimes H^0$ which has constant multiplicity $\dim H^0$, while in the twisted case, we refer to \cite[Corollary~18.10]{Ca}, which explicitly describes two multiplicities in cases $\widetilde{\mathrm{B}}^\text{t}_\ell, \widetilde{\mathrm{C}}^\text{t}_\ell, \widetilde{\mathrm{F}}^\text{t}_4$ and $\widetilde{\mathrm{G}}^\text{t}_2$ that occur by a period equal to the order of $\tau$, and to \cite[Corollary~18.15]{Ca}, which says the multiplicity of $k\delta$ is $\ell$ in $\widetilde{\mathrm{C}}'_{\ell}$, and $1$ in $\widetilde{\mathrm{A}}'_1$.
Part (iv) is due to the fact that $\g_{k\delta}$ is always a subspace of $t^k \otimes H^0$ and we refer to a case-by-case discussion in \cite[Theorems~18.5, 18.9, and 18.14]{Ca}.
Then the definition of the Lie bracket on $\g \subseteq \hat{\mathfrak{L}}(L^0)$ (in \cite[Section~18.1]{Ca}) tells us that
\[
[t^k \otimes H^0, t^\ell \otimes H^0] = k\delta_{k, -\ell}\la H^0, H^0 \ra c = 0
\]
for $k\neq -\ell$.
\end{proof}

\comm{
\begin{theorem}
Let $A$ be a symmetrizable Cartan matrix, $\b \in \Delta_{\im}$ be isotropic, and $\a\in \Bar{\Delta}^+$ be such that $|R_\a(\b)|>1$. Suppose $(\a, \b) = 0$. Then $R_\a(\b)$ is always bi-infinite, and
\en
\item If $\a$ is real, $R_\a(\b)$ consists only of real roots which have multiplicity 1.
\item If $\a$ is imaginary, $R_\a(\b)$ consists of imaginary roots only,
and the multiplicities take on at most 3 values, at most 2 of which occur periodically.
\te
\end{theorem}}

\begin{theorem*}[\ref{thm:isoZero}]
Let~$\Delta$ be the set of roots of a symmetrizable Kac--Moody algebra $\g$. 
Let $\b \in \Delta_{\im}$ be isotropic, and $\a\in \Bar{\Delta}$ be such that $|R_\a(\b)|>1$. Suppose $(\a, \b) = 0$. Then $R_\a(\b)$ is always bi-infinite, and
\begin{enumerate}[label=(\roman*)]
\item If $\a$ is real, $R_\a(\b)$ consists only of real roots which have multiplicity 1.
\item If $\a$ is imaginary, $R_\a(\b)$ consists only of imaginary roots,
and their multiplicities take at most 3 values, at most 2 of which occur periodically.
\end{enumerate}
\end{theorem*}
\begin{proof}
By Equation~(\ref{eqn:dim-=+}) and Equation~(\ref{eqn:wR=Rw}), without loss of generality, we may assume $\a\in \Bar{\Delta}^+$ and $\b\in K$.
Let $J = \supp(\b)$. By \cite[Proposition~16.29]{Ca}, $A_J$ is of affine type. 
Then by Proposition~\ref{prop:equalDim}, it suffices to consider $\g(A_J)$.
Take $A = A_{J}$ in Proposition~\ref{prop:affResults}. Thus any imaginary $\b$ must be a multiple of $\delta$.
By Proposition~\ref{prop:affResults}(i) and (ii), $R_\a(\b)$ are bi-infinite.
Now part (i) follows from Proposition~\ref{prop:Wprops}(iii).
If both $\a$ and $\b$ are imaginary, then both are multiples of $\delta$. Thus part (ii) follows from Proposition~\ref{prop:affResults}(iii), where the third value comes from the possibility that $0\in R_\a(\b)$, which has multiplicity $\dim \h$.
\end{proof}

\subsection{Negative inner product} \label{subsec:negProd}
We now assume that $\b \in \Delta_{\im}$ is isotropic and $\a\in \Bar{\Delta}$ with $(\a, \b)<0$. Note $\a \neq 0$ in this case. 
In contrast to the non-isotropic case, we cannot use Corollary~\ref{cor:MarCorC} or \cite[Corollary~9.12(a)]{K} when $(\b, \b) = 0$ as we did in the last section,
nor can we focus on an affine subalgebra as in the last subsection since we do not have $\supp(\a) \subseteq \supp(\b)$ anymore.
Nonetheless, there exists a Heisenberg subalgebra $H(\b)$ of $\g$, constructed below, that allows us to consider its irreducible module in $\g$.
It turns out that this module is contained in $\mathfrak{S}_\a(\b)$, and is  isomorphic to a certain induced module.
The main result (Theorem~\ref{thm:isoNeg}) is proved by using the Poincar\'e--Birkhoff--Witt theorem applied to this induced module.
\begin{lemma}\label{lem:notBiInf}
For $\a\in \Delta^+$ such that $(\a,\b)<0$, there exists $k>0$ such that
\[
\a-k\b\notin \Delta
\]
and $R_\a(\b)\subseteq \Delta^+$.
In particular, $R_\a(\b)$ is finite in the direction of $-\b$.
\end{lemma}
\begin{proof}
Recall that both $\a$ and $\b$ are assumed to be positive. 
By Proposition~\ref{prop:affineSupp}, $\supp(\a)$ contains some index $i\notin \supp(\b)$. 
Thus for some large $k$, the coefficient of $\a_i$ in $\a-k\b$ is positive while the coefficients of $\a_j$ for $j\in \supp(\b)$ are negative, which means $\a-k\b \notin \Delta$.
\end{proof}

In contrast to  the assertion in Lemma~\ref{lem:notBiInf} about the direction of $-\b$, we will show in Theorem~\ref{thm:isoNeg} that  $R_\a(\b)$ is infinite in the direction of $\b$ and the growth along $R_\a(\b)$ has a subexponential lower bound. 
To prove Theorem~\ref{thm:isoNeg}, we will use the following observation.
By Proposition~\ref{prop:propIm} we have $\dim \g_{k\b}\ge 1$, for $k\in \Z$. Then, by Proposition~\ref{prop:(h)}, for each $k\in \Z_{>0}$, we can choose $x_k\in \g_{k\b}$ and $y_k\in \g_{-k\b}$ such that $(x_k,y_k)=1/k$.
Let $c=\Lambda^{-1}(\b)\in \h\subseteq \g$. We define
\[
H(\b) :=\C c \oplus \bigoplus\limits_{i\ge 1} \C x_i \bigoplus\limits_{j\ge 1} \C y_j\subseteq \g.
\]
Observe that for $i, j\in \Z_{>0}$: 
\begin{enumerate}[label=(\roman*)]
    \item We have $[c, x_i] = \la i\b, c \ra x_i= (i\b, \b)x_i = 0$ (similarly for $y_j$), and thus $c$ is central.
    \item For $i\neq j$, since $[\g_{i\b}, \g_{j\b}]=0$ by Proposition~\ref{prop:affResults}(iv), we have $[x_i,x_j]=[y_i,y_j]=0$.
    \item Since $[x_k, y_k] = (x_k, y_k)\Lambda^{-1}(k\b) = c$, we have $[x_i,y_j]=\delta_{ij}c$.
\end{enumerate}
Such Lie algebras are known as \textit{Heisenberg algebras}.
We further define
\[
H(\b)^- := \C c \oplus \bigoplus\limits_{i\ge 1} \C y_i\subseteq H(\beta).
\]
Let $\kappa = (\a, \b)<0$ and $\C_\kappa$ be the one-dimensional $H^-(\b)$-module where $c\cdot1 = \kappa$ and $y_i\cdot1=0$.

\begin{theorem*}[\ref{thm:isoNeg}]
Let~$\Delta$ be the set of roots of a symmetrizable Kac--Moody algebra $\g$. 
Let $\b \in \Delta_{\im}$ be isotropic, and $\a\in \Bar{\Delta}$ be such that $|R_\a(\b)|>1$. Suppose $(\a, \b) \neq 0$. 
Then $R_\a(\b)$ is semi-infinite, and the growth along $R_\a(\b)$ has a subexponential lower bound.
\end{theorem*}
\begin{proof}
Suppose $(\a, \b)\neq 0$. First, if $\a \in \Delta^-$, by Equation~(\ref{eqn:dim-=+}) and Equation~(\ref{eqn:R=-R}), we may replace $\a$ by $-\a \in \Delta^+$. Thus it suffices to consider the case $\a \in \Delta^+$.
For $\b$, by Equation~(\ref{eqn:wR=Rw}), we may assume that $\b \in \Delta^+_{\im}$. By Proposition~\ref{prop:Wprops}(ii), Proposition~\ref{prop:K_A}, and Proposition~\ref{prop:(h*)}(i), we may further assume $\b\in K$.
Therefore, we must have $(\a, \b) < 0$ by $(\a, \b)\neq 0$ and the definition of $K$.

By Lemma~\ref{lem:notBiInf}, replacing $\a$ by $\a-N\b$ for some $N>0$ if needed, we may assume that $\a-i\b\notin \Delta$ for all $i>0$. Let $v\in \g_\a$. Then $[y_i,v]=0$ for any $y_i$, and $[c,v]=\kappa v$ where $\kappa=(\a,\b)<0$. Thus, under the adjoint action $\C v \cong \C_\kappa$.
We further let $\g_v$ be the $H(\b)$-module generated by $v$, that is, 
\[
\g_v :=\mathfrak{U}(H(\b))\cdot v\subseteq \g.
\]
\comm{
By construction, the non-trivial $H(\b)$-module $\g_v$ is isomorphic to a quotient of
\[
\ind_{H(\b)^-}^{H(\b)} \C_\kappa
\]
using the universal property of induced modules.
But this module is irreducible by \cite[Proposition~1.7.2]{FLM} (note $\kappa \neq 0$, which is required). 
Therefore,
\[
\g_v\cong \ind_{H(\b)^-}^{H(\b)} \C_\kappa  \cong \mathfrak{U}(H(\b))\otimes_{H(\b)^-} \C v.
\]
}

Consider the induced module
\[
I_\kappa := \ind_{H(\b)^-}^{H(\b)} \C_\kappa \cong \mathfrak{U}(H(\b))\otimes_{H(\b)^-} \C v.
\]
Since $\C_\kappa \cong \C v$, we have an $H(\b)^-$-map from $\C_\kappa$ to $\g_v$. Then by the universal property of $I_\kappa$, there exists a unique $H(\b)$-map from $I_\kappa$ to $\g_v$. Thus $\g_v$ is a quotient of $I_\kappa$. By \cite[Proposition~1.7.2]{FLM}, $I_\kappa$ is irreducible and therefore we have
\[
\g_v \cong \mathfrak{U}(H(\b))\otimes_{H(\b)^-} \C v.
\]
We remark that our $x_i$ and $y_j$ correspond to $y_j$ and $x_i$ in \cite[Section~1.7]{FLM}. Since the construction of Heisenberg algebras is symmetric with respect to these generators, the result holds. 

By the Poincar\'e--Birkhoff--Witt theorem, we have the following basis for $\g_v$
\[
\{x_{i_1}\cdots x_{i_s}\cdot v \mid  s\ge 0, i_1\le \dots \le i_s\},
\]
where the product is understood as the corresponding iterated Lie bracket. 
Additionally, as a subspace of $\mathfrak{S}_\a(\b)$, $\g_v$ has an $\N$-grading. Explicitly, $\g_v = \bigoplus_{k\in \N}\g_v^k$, where each $\g_v^k$ is spanned by 
\begin{equation} \label{eqn:kPart}
    \left \{x_{i_1}\cdots x_{i_s}\cdot v\mid i_1\le \dots \le i_s, \sum i_j = k\right \}.
\end{equation}
Note that $x_k\in \g_{k\b}$ acts as a raising operator on $\g_v$, adding $k\b$ to the weight spaces of $\g_v$.
In particular, $\g_v^k \subseteq \g_{\a+k\b}$, meaning $\g_{\a+k\b}$ is non-zero, and $\a+k\b$ is a root.
This shows $R_\a(\b)$ is infinite in the direction of $\b$.
Furthermore, by Equation~(\ref{eqn:kPart}), each subspace $\g_v^k$ has dimension $p(k)$, where $p(k)$ counts the number of partitions of $k$. Hence
\[
\dim \g_{\a+k\b}\ge \dim \g_v^k =  p(k).
\]
The famous asymptotic formula
\[
p(n)\sim {\frac {1}{4n{\sqrt {3}}}}\exp \left({\pi {\sqrt {\frac {2n}{3}}}}\right)
\]
by Hardy and Ramanujan \cite{HRpartFun} is of growth $\mathcal{E}_{1/2}$ (see Section~\ref{subsec:RS&G}).
This is clearly superpolynomial but less than the exponential growth $\mathcal{E}_1$. 
Thus, the growth along $R_\a(\b)$ has a lower bound that is subexponential.
\end{proof}

\section{A general linear bound for root multiplicities} \label{sec:linBound}

The previous two sections investigated the asymptotic behavior of root strings in the direction of an imaginary root. 
In this section, we prove the following general linear bound on root multiplicities locally. 
Let $\g  = \g(A)$ for a symmetrizable generalized Cartan matrix $A$ as before.
The main result Theorem~\ref{thm:dimGrowth} is a direct application of Theorem~\ref{thm:MarThmA}, combined with Lemma~\ref{lem:linBound} purely in the nature of linear algebra.

\begin{theorem*}[\ref{thm:dimGrowth}] 
Let $\g = \g(A)$ for a symmetrizable generalized Cartan matrix $A$.
If $\a$ and $\b$ are distinct roots of $\g$ such that $(\a,\b)<0$, then 
$$
\dim \g_{\a+\b}\ge \dim \g_\a+ \dim \g_\b -1.
$$
\end{theorem*}

To prove Theorem~\ref{thm:dimGrowth}, we begin with the following result.

\begin{lemma}\label{lem:no_pure}
Let $U$ and $V$ be finite dimensional vector spaces over $\C$ of 
dimensions $m,n$ respectively and let $L$ be a 
subspace of $U\otimes V$. Then if $L$ 
contains no simple tensors, we must have
\[
\dim L\le (m-1)(n-1).
\]
\end{lemma}
\begin{proof}
 We choose bases for $U$ and $V$ and a dual basis for $U^*$. Let $\psi$ denote the isomorphism from $U \otimes V$ to the space of $m\times n$ complex matrices $M_{m, n}$ with respect to these bases. Explicitly we have $\End(U, V) \cong U^*\otimes V \cong U\otimes V$.

Let $R\subseteq M_{m,n}$ be the subset of matrices of rank $\le 1$.
Then $R$ is the image of the set of simple tensors in $U\otimes V$ under $\psi$. 
Let $L'=\psi(L)$. We must show that if $\dim L = \dim(L')>(m-1)(n-1)$ then $L'\cap R \ne 0$.

If $A = (A_{ij}) \in M_{m,n}$ and $A\in R$, all rows must be multiples of each other. This is equivalent to:
\[
A_{i,j}A_{1,1}-A_{i,1}A_{1,j}=0, \quad \text{for $2\le i \le m$, $2\le j \le n$,}
\]
which are $(m-1)(n-1)$ homogeneous quadratic equations that determine membership in $R$.

On the other hand, $L'\subseteq M_{m,n}$ is a subspace determined by $mn-{\rm dim} L'$ homogeneous linear equations.
So $L'\cap R$ consists of the points satisfying the $(m-1)(n-1)$ homogeneous polynomials for membership in $R$ and the $mn-{\rm dim}(L')$ homogeneous polynomials for membership in $L'$. Suppose $\dim L' \geq (m-1)(n-1)+1$. Then in total, there are at most
\[
(m-1)(n-1)+mn-{\rm dim}(L') \leq (m-1)(n-1)+mn-(m-1)(n-1)-1=mn-1
\]
homogeneous polynomial equations in $mn$ variables over $\C$ that determine this membership.
By Bezout's Theorem \cite[p. 246]{ShaAlgGeo}, there must be a non-zero solution.
\end{proof}

\begin{lemma}\label{lem:linBound}
Let $U,V,W$ be vector spaces and $m={\rm dim}(U)$, $n= {\rm dim}(V)$. Let $f:U\otimes V \rightarrow W$ be a bilinear map 
such that for all $u\in U\setminus \{0\}, v\in V\setminus \{0\}$, $f(u,v)\ne 0$.
Then
\[
\dim W\ge m+n-1.
\]
\end{lemma}
\begin{proof}
By our assumptions on $f$, $\ker f$ 
cannot have any simple tensors $u\otimes v$, where $u,v\ne 0$. 
Then by Lemma~\ref{lem:no_pure}, 
\[
\dim \ker f\le (m-1)(n-1).
\]
By the rank--nullity theorem, 
\[
\dim \mathrm{Im}(f)= \dim U\otimes V - \dim \ker f \ge mn-(m-1)(n-1) = m+n-1.
\]
Therefore $\dim W \ge \dim \mathrm{Im}(f) \ge m+n-1$.
\end{proof}

\begin{proof}[Proof of Theorem~\ref{thm:dimGrowth}]
By Theorem~\ref{thm:MarThmA}, the Lie bracket $[\cdot, \cdot]$ is a bilinear map satisfying the conditions for Lemma~\ref{lem:linBound}.
This allows us to conclude that if $\a_1,\a_2\in \bar{\Delta}$ are distinct, then $\dim \g_{\a_1+\a_2}\ge \dim \g_{\a_1}+\dim \g_{\a_2}-1$, proving Theorem~\ref{thm:dimGrowth}.
\end{proof}

By Proposition~\ref{prop:semiinfRS}, for any non-isotropic root $\b$, the root string $R_\a(\b)$ is always infinite if $|R_\a(\b)|>1$. 
We cannot have more than one root $\gamma \in R_\a(\b)$ which satisfies $(\gamma,\b)=0$.
Suppose $(\gamma, \b) = 0$, then $(\gamma+s\b,\b)=s(\b,\b)$ cannot be constant for $s\in\Z$ if $|R_\a(\b)|>1$.
So $R_\a(\b)$ will always contain a $\gamma$ such that $(\b,\gamma)\ne 0$.
Now, by inductively applying Theorem~\ref{thm:dimGrowth}, we derive the following result.

\begin{corollary}\label{cor:increase}
Let $\b \in \Delta_{\im}$ be non-isotropic and let $\a \in \bar{\Delta}$ be such that $|R_\a(\b)|>1$. 
Then $R_\a(\b)$ is infinite in at least one direction. Moreover,
\en[label=(\roman*)]
\item If there is some root $\gamma \in R_\a(\b)$ with $(\b,\gamma)<0$, then $R_\a(\b)$ is infinite in the direction of~$\beta$. If 
$\dim \g_\b >1$, then $\dim \g_{\gamma+i\b}$ is strictly increasing in $i$.
\item If there is some root $\gamma \in R_\a(\b)$ with $(\b,\gamma)>0$,  then $R_\a(\b)$ the root string is infinite in the direction of~$-\beta$. If $\dim \g_\b >1$, then $\dim \g_{\gamma-i\b}$ is strictly increasing  in $i$.

\te
\end{corollary}

Corollary~\ref{cor:increase} can also be used to determine if the dimensions along bi-infinite root strings are increasing.
 That is, if $(\beta, \gamma)<0$ and $\gamma - \beta$, $\gamma - 2\beta,\dots, \gamma - \left\lceil{\frac{(\beta,\gamma)}{(\beta,\beta)}}\right\rceil\beta$ are roots, then $(\beta, \gamma - \left\lceil{\frac{(\beta,\gamma)}{(\beta,\beta)}}\right\rceil\beta)> 0$ and by Corollary~\ref{cor:increase}(ii), $R_\a(\b)$ is also infinite in the direction of $-\beta$ so $R_\a(\b)$ is bi-infinite. Additionally, $\dim \g_{\a+i\b}$ is decreasing for $i \leq - \left\lceil{\frac{(\beta,\gamma)}{(\beta,\beta)}}\right\rceil$, and increasing for $i \geq 1$.
These statements are similarly true if $(\beta,\gamma)>0$ and $\gamma + \beta, \gamma + 2\beta, \dots, \gamma + \left\lceil{-\frac{(\beta,\gamma)}{(\beta,\beta)}}\right\rceil\beta$ are roots.

As another application of Theorem~\ref{thm:MarThmA}, the existence of semi-infinite root strings in the direction of an imaginary root for any symmetrizable Kac--Moody algebra is always guaranteed.
\begin{corollary} \label{cor:ind}
Let $\g$ be a symmetrizable Kac--Moody algebra with root system $\Delta$ and symmetric invariant bilinear form $(\cdot,\cdot)$. Let $\a,\b\in \Delta$ be distinct with $(\a,\b)<0$. 
Then $\a+\b$ is a root. If in addition, $\b$ is imaginary, then $\a+\N \b \subseteq R_{\a}(\b)$ with $\dim \g_{\a+m\b}-\dim \g_{\a+n\b}
\ge (m-n)(\dim \g_{\b}-1)$ for any $m>n$.
\end{corollary}
\begin{proof} By Theorem~\ref{thm:MarThmA}, $[x_1,x_2]\ne 0$ so $\g_{\a+\b}\neq \{0\}$ and thus $\a+\b$ is a root. 
Suppose $(\b,\b)\le 0$. Then as $(\a+k\b, \b) \le (\a, \b) < 0$, inductively we can show that $\a+k\b$ is a root for any $k > 0$.
By Theorem~\ref{thm:dimGrowth}, again as $(\a+k\b, \b) \le (\a, \b) < 0$,
we have $\dim \g_{\a+(k+1)\b}-\dim\g_{\a+k\b}\ge \dim \g_{\b}-1$ for all $k>1$. 
The assertion then follows inductively. 
\end{proof}

\bibliographystyle{amsalpha}
\bibliography{ref}

\end{document}